\documentclass[12pt]{amsart}

\usepackage{amsmath}
\usepackage{amscd}
\usepackage{eucal}
\usepackage{amssymb}
\usepackage{epic}
\usepackage{graphics}
\usepackage{color}
\usepackage{mathrsfs}
\usepackage{multirow}
\usepackage{float}
\usepackage{graphicx}
\usepackage{mathtools}
\usepackage{verbatim}
\usepackage{xypic}
\usepackage{caption}

\numberwithin{equation}{section}
\numberwithin{table}{section}

\theoremstyle{plain}
\newtheorem{theorem}{Theorem}[section]
\newtheorem*{thm:A}{Theorem \ref{thm:A}}
\newtheorem*{thm:B}{Theorem \ref{thm:B}}
\newtheorem{lemma}{Lemma}[section]

\theoremstyle{remark}

\newcommand{\ZpL}[1]{\mathbb{Z}_{p}^{\mathcal{L}_{{#1}}}}

\newcommand{\LL}{{\mathcal{L}}}
\newcommand{\HH}{{\mathcal{H}}}
\newcommand{\Zp}{{\mathbb{Z}_{p}}}
\newcommand{\Qp}{{\mathbb{Q}_{p}}}

\newcommand{\Fq}{{\mathbb{F}_{q}}}

\newcommand{\leftsub}[2]{{\vphantom{#2}}_{#1}{#2}}
\newcommand{\qbinom}[2]{{\left[\begin{smallmatrix}#1\\#2\end{smallmatrix}\right]_q}}
\newcommand{\qqbinom}[2]{{\left[\begin{smallmatrix}#1\\#2\end{smallmatrix}\right]_{q^2}}}
\newcommand{\abs}[1]{\lvert#1\rvert}

\newcommand{\Qell}{{K}}
\newcommand{\Fell}{{{F}}}
\newcommand{\Rell}{{R}}
\newcommand{\allone}{{\mathbf{1}}}

\DeclareMathOperator{\Image}{Im}
\DeclareMathOperator{\Ker}{Ker}
\DeclareMathOperator{\soc}{soc}

\DeclareMathOperator{\GL}{{\mathrm{GL}}}
\DeclareMathOperator{\PG}{{\mathrm{PG}}}
\DeclareMathOperator{\End}{{\mathrm{End}}}

\title[Grassmann graphs]{The Smith group and the critical group of the Grassmann graph of lines
in finite projective space and of its complement}
\author[Ducey]{Joshua E. Ducey}
\author[Sin]{Peter Sin}

\address{Dept.\ of Mathematics and Statistics, James Madison University, Harrisonburg, VA 22807, USA}
\email{duceyje@jmu.edu}
\address{Dept.\ of Mathematics, University of Florida, P.O. Box 118105, Gainesville, FL 32611, USA}
\email{sin@ufl.edu}
\thanks{This work was partially supported by a grant from the Simons Foundation (\#204181 to Peter Sin). }
\thanks{Research was (partially) completed while the second author was visiting the Institute for Mathematical Sciences, National University of Singapore in 2016.}

\keywords{invariant factors, elementary divisors, Smith normal form, critical group, Jacobian group, sandpile group, adjacency matrix, Laplacian, Grassmann}
\subjclass[2010]{05C50}

\begin{document}

\begin{abstract}
We compute the elementary divisors of the adjacency and Laplacian matrices of the Grassmann graph on $2$-dimensional subspaces in a finite vector space.
We also compute the corresponding invariants of the complementary graphs.
\end{abstract}
\maketitle
\section{Introduction}
The {\it Smith group} of a graph is the abelian group defined by 
using an adjacency matrix of the graph as a relation matrix. The name
comes from the fact that a cyclic decomposition of this group is given by the Smith
normal form of the adjacency matrix. If we use the Laplacian matrix
of the graph as a relation matrix then the finite part of the abelian group
defined is called the the {\it critical group}. By the well-known
Matrix Tree Theorem of Kirchhoff, the order of the critical group of a connected graph
is equal to the number of spanning trees in the graph. This group also arises
in chip-firing games on graphs and in the closely-related sandpile model in physics, 
where it is also known as the sandpile group. It is therefore of some interest
to compute the  Smith groups and critical groups for individual graphs
or for families of graphs. A particular class that has proved amenable
to computations  is the class of strongly regular graphs (e.g. \cite{Paley}). 
In this paper we treat these questions for 
the Grassmann graph defined by the incidence of lines in a finite
projective space $\PG(n-1,q)$, as well as its complementary graph, the skew lines
graph.
As well as being strongly regular these graphs admit the action of the
general linear group of the projective space, which allows the application of 
extensive machinery from representation theory. The problem of computing
the Smith normal forms or, equivalently, the elementary divisors
of the adjacency and Laplacian matrices, falls naturally into two parts, 
one for the characteristic $p$ of the projective space and the second for
all other primes.  In \S\ref{genlemmas}, we establish some general lemmas
needed in both parts. 
The charactersistic $p$  case was handled in projective dimension 3
in an earlier paper \cite{skew}, by making  use of the computations
of the $p$-elementary divisors of point-subspace incidence matrices in
\cite{CSX}. Fortunately, it is possible to extend 
the method to higher dimensions without great difficulty, and this is 
done in \S\ref{charp}.
We then turn to the cross-characteristic case.
Here we rely heavily on James's theory \cite{James} 
of unipotent representations of $\GL(n,q)$,
over a field of characteristic $\ell\neq p$. From general considerations we know that
the multiplicity of a given power $\ell^a$ as an elementary divisor
is the dimension of some $\GL(n,q)$-subquotient of the $\ell$-modular permutation module on lines, and detailed information about the structure of the 
submodule lattice of this permutation module allows us to determine all 
such multiplicities.  In \S\ref{crosschar} we collect
together the required results from the $\ell$-modular representation theory
of $\GL(n,q)$. The  detailed computations are then made in the following four sections,
treating in turn the critical group and the Smith group of the Grassmann graph, 
followed by the critical group and the Smith group of the skew lines graph. 
Finally, we include an appendix
displaying all the possible submodule structures of the $\ell$-modular permutation module,
corresponding to various arithmetic conditions satisfied by $\ell$ and $n$.
Although we never use these informal pictures in our proofs,
we have found them very helpful in navigating through the many cases into which
the computations are divided.

\section{Parameters of the graphs}
Let $V$ be an $n$-dimensional vector space $(n \geq 4)$ over a finite field $\Fq$ of $q=p^{t}$ elements.  Consider the graph with vertex set the $2$-dimensional subspaces of $V$ (or $2$-spaces for short), where a pair of $2$-spaces are adjacent when their intersection is trivial.  This is the skew-lines graph of the associated projective space, which we denote by $\Gamma$.  Under some ordering of the vertices, we let $A$ denote the adjacency matrix. 

The complement of this graph is the Grassmann graph $\Gamma^{\prime}$ on $2$-spaces, and we denote by $A^{\prime}$ its adjacency matrix.

These graphs are strongly regular \cite[Chapter 8]{BH}.  The parameters $v^{\prime}, k^{\prime}, \lambda^{\prime}, \mu^{\prime}$ of $\Gamma^{\prime}$ are given below.  We use the $q$-binomial coefficients 
\[\qbinom{m}{k} = \frac{(q^{m} - 1)(q^{m-1}-1) \cdots (q^{m-k+1} - 1)}{(q^{k} - 1)(q^{k-1} - 1) \cdots (q-1)}
\]
 which count the number of $k$-dimensional subspaces in an $m$-dimensional vector space over $\Fq$.
\begin{align*}
v^{\prime} &= \qbinom{n}{2} \\
k^{\prime} &= q(q+1)\qbinom{n-2}{1} \\
\lambda^{\prime} &= \qbinom{n-1}{1} + q^{2} - 2\\
\mu^{\prime} &= (q+1)^{2}.
\end{align*}

Let $v, k, \lambda, \mu$ denote the parameters of the skew-lines graph $\Gamma$.  Since this graph is complementary to $\Gamma^{\prime}$ we have
\begin{align*}
v &= v^{\prime} \\
k &= v - k^{\prime} - 1 = q^{4} \qbinom{n-2}{2}  \\
\lambda &= v - 2k^{\prime} + \mu^{\prime} - 2 = 2k + \mu - v\\
\mu &= v - 2k^{\prime} + \lambda^{\prime}.
\end{align*}
We also record that $A$ has spectrum 
\[
q^{4} \qbinom{n-2}{2},  \, -q^{2} \qbinom{n-3}{1}, \, q
\]
with respective multiplicities $1, \qbinom{n}{1}-1, \qbinom{n}{2}-\qbinom{n}{1}$.

Let $L=kI-A$ be the Laplacian matrix of $\Gamma$ and $L'=k'I-A'$ that
of $\Gamma^{\prime}$. We note that $A$ and $A'$ are invertible matrices,
while $L$ and $L'$ have corank 1.
We shall make use of relations satisfied by $A$, $L$, $A'$ and $L'$
that follow from the strong regularity of $\Gamma$ and $\Gamma'$. 

These are well known \cite[Theorem 8.1.2]{BH}.
\begin{lemma}\label{srgeq} Let $\tilde A$ be the adjacency matrix of a strongly
regular graph with parameters $(\tilde v,\tilde k,\tilde\lambda, \tilde\mu)$
and restricted eigenvalues $\tilde r$ and $\tilde s$, and let
$\tilde L={\tilde k}I-\tilde A$ be the Laplacian matrix. Then 
\begin{enumerate}
\item[(i)] ${\tilde A}^2+(\tilde\mu-\tilde\lambda)\tilde A+(\tilde\mu-\tilde k)I=\tilde\mu J$, where $J$ is the matrix all of whose entries are $1$.
\item[(ii)] $(\tilde A-\tilde rI)(\tilde A-\tilde sI)=\tilde\mu J$, and 
$(\tilde L-(\tilde k-\tilde r)I)(\tilde L-(\tilde k-\tilde s)I)=\tilde\mu J$. 
\end{enumerate}
\end{lemma}

\section{General results on modules}\label{genlemmas}

Let $(R,(\pi))$ be a local PID with residue field $R/(\pi)=F$, field of quotients $\Qell$, and let $M$ a
finitely generated free $R$-module.
Let $\overline{M}=M/\pi M \cong F\otimes_RM$,
and for  any submodule $U$ of $M$, we denote by $\overline U$ its image
in $\overline M$.
Let $\phi\in\End_R(M)$, and for $i\geq 0$ let 
\begin{equation}
M_{i} = M_i(\phi)=\{n\in M|\phi(n)\in \pi^iM\}.
\end{equation}
Then we have
\begin{equation}\label{Mchain}
M=M_0\supseteq M_1\supseteq\cdots\supseteq \Ker(\phi),
\end{equation}
and
\begin{equation}\label{Mbarchain}
\overline M=\overline M_0\supseteq \overline M_1\supseteq\cdots\supseteq \overline{\Ker(\phi)}.
\end{equation}
Then $e_i=e_i(\phi):=\dim_\Fell(\overline{M}_i/\overline{M}_{i+1})$ is the multiplicity
of $\pi^i$ as an elementary divisor of $\phi$. 
($e_0$ is the rank of $\overline\phi$.)
 We will often use the formula
\begin{equation}\label{Mbardim}
\dim_\Fell\overline{M}_i=\dim_\Fell\overline{\Ker(\phi)}+\sum_{j\geq i}e_j.
\end{equation}

Note that $\overline{\Ker(\phi)}$ is not, in general, equal to the kernel of the induced map $\overline\phi\in\End_\Fell(\overline{M})$. Indeed, $\Ker(\overline\phi)=\overline{M}_1$.

One can also define an ascending chain of submodules of $M$:
\[
N_{i} = N_{i}(\phi) = \{\pi^{-i}\phi(n) \vert n \in M_{i}\}.
\]
A formula similar to the above holds, namely,
\begin{equation}\label{Nbardim}
\dim_{\Fell}\overline{N}_{i} = \sum_{j=0}^{i}e_{j}.
\end{equation}

\begin{lemma}\label{eldiv} Let $M$, and $\phi$ be as above.
Let $d$ be the $\pi$-adic valuation of the product of the nonzero
elementary divisors of $\phi$, counted with multiplicities.
Suppose  that in (\ref{Mbarchain}) we have an increasing sequence of indices 
$0<a_1<a_2<\cdots<a_h$ and a corresponding sequence of lower bounds
$b_1>b_2>\cdots>b_h$ satisfying the following conditions. 
\begin{enumerate}
\item$\dim_F\overline M_{a_j}\geq b_j$ for $j=1$,\ldots, $h$. 
\item $\sum_{j=1}^{h}(b_j-b_{j+1})a_j=d$, where we set $b_{h+1}=\dim_F\overline{\Ker(\phi)}$. 
\end{enumerate}
Then the following hold.
\begin{enumerate}
\item $e_{a_j}(\phi)=b_j-b_{j+1}$ for $j=1$,\ldots, $h$. 
\item $e_0(\phi)=\dim_F\overline M-b_1$.
\item $e_i(\phi)=0$ for $i\notin\{0, a_1,\ldots, a_h\}$.
\end{enumerate}
\end{lemma}
\begin{proof}
We have
\begin{equation}
\begin{aligned}
d&=\sum_{i\geq0}ie_i\\
&\geq\sum_{j=1}^{h-1}\left(\sum_{a_j\leq i<a_{j+1}}ie_i\right)+\sum_{i\geq a_h}ie_i\\
&\geq\sum_{j=1}^{h-1}\left(a_j\sum_{a_j\leq i<a_{j+1}}e_i\right)+a_h\sum_{i\geq a_h}e_i\\
&=\sum_{j=1}^{h-1}a_j(\dim_F\overline M_{a_j}-\dim_F\overline M_{a_{j+1}})+a_h(\dim_F\overline M_{a_h}-\dim_F\overline{\Ker(\phi)})\\
&=a_1\dim_F\overline M_{a_1}+\sum_{j=2}^{h}(a_j-a_{j-1})\dim_F\overline M_{a_j}-a_h\dim_F\overline{\Ker(\phi)}\\
&\geq a_1b_1+\sum_{j=2}^{h}(a_j-a_{j-1})b_j-a_hb_{h+1}\\
&=\sum_{j=1}^{h}(b_j-b_{j+1})a_j=d.
\end{aligned}
\end{equation}
Therefore we must have equality throughout. The equality of the first three
lines implies that $e_i=0$ for $i\notin\{0,a_1,\ldots,a_h\}$, proving (3).
It follows that for $j=1$,\ldots, $h-1$ we have $e_{a_j}=\dim_F\overline M_{a_j}-\dim_F\overline M_{a_{j+1}}$, and $e_{a_h}=\dim_F\overline M_{a_h}-\dim_F\overline{\Ker(\phi)}$.
The equality of the second and third last rows then shows that 
inequalities $\dim_F\overline M_{a_j}\geq b_j$ are all equalities, for $j=1$,\ldots, $h$,
so $e_{a_j}=(b_j-b_{j+1})$ and (1) is proved.
Finally, $e_0(\phi)$ is the rank of the induced endomorphism
$\overline\phi\in\End_F(\overline M)$, whose kernel is $\overline M_1$ by definition, and our discussion shows that $\overline M_1=\overline M_{a_1}$, proving(2).
\end{proof}

Let $M_{\Qell}=\Qell\otimes_\Rell M$.  The following lemma is well known (cf. \cite[Proposition 12.8.3]{BH}).
\begin{lemma}\label{eigenspacebound} Let $\phi\in\End_\Rell(M)$.
Suppose the induced endomorphism $\phi_{\Qell}$ 
of $M_{\Qell}$ has an eigenvalue $u\in \Rell$. Let $d$ be the dimension of the eigenspace
\begin{equation*}
(M_{\Qell})_u=\{m\in M_{\Qell} \mid \phi_{\Qell}(m)=um \}.
\end{equation*}
Then $d\leq \dim_\Fell\overline{M}_b$, for any $b$ with $\pi^b\mid u$.
\end{lemma}
\begin{proof}
The submodule $M\cap (M_{\Qell})_u$ is pure (\cite[p.84]{Curtis-Reiner}), has rank
equal to $d$ and  clearly lies in $M_b$.
It follows that $d=\dim_{\Fell}\overline{M\cap (M_{\Qell})_u}\leq\dim_\Fell\overline{M}_b$.
\end{proof}

Assume that $M$ is an $\Rell G$-module for a group $G$ and 
$\phi$ an $\Rell G$-module homomorphism. Then we
can strengthen the conclusion of Lemma~\ref{eigenspacebound} to a statement
about the $\Fell G$-composition factors of $\overline{M_b}$.

\begin{lemma}\label{decompmap} With the notation above and Lemma~\ref{eigenspacebound},
let $E$ be a $\Qell G$-submodule of $(M_{\Qell})_u$. 
Then the following hold.
\begin{enumerate}
\item[(i)]
$M\cap E$ is a pure submodule of $M$ such that $\Qell\otimes_\Rell(M\cap E)\cong E$.
\item[(ii)] If $X$ be any $\Rell$-free $\Rell G$-module with $\Qell\otimes_\Rell X\cong E$, then $X/\pi X$ and $\overline{M\cap E}$ have the same $\Fell G$-composition factors, counting multiplicities.
\item[(iii)] For any $b$ with $\pi^b\mid u$, the $\Fell G$-composition factors
of $\overline{M_b}$ include those of $X/\pi X$, counting multiplicities.
\end{enumerate}
\end{lemma}
\begin{proof} Part (i) is immediate from the definitions of pure submodule (\cite[p.84]{Curtis-Reiner}) and $M\cap E$. By (i), the image $\overline{M\cap E}$ of 
$M\cap E$ in $\overline{M}$ is isomorphic to $(M\cap E)/\pi(M\cap E)$, and this has the same composition factors as $X/\pi X$ since $X$ and $M\cap E$ 
are both $R$-forms of $E$, by a general principle
\cite[Proposition 16.16]{Curtis-Reiner} in modular representation theory. Thus (ii)
holds.  To see (iii), we note that from the hypothesis on $b$  we have 
$\overline{M\cap E}\subseteq\overline{M}_b$, and the statement about
composition factors holds by (ii). 
\end{proof}


\section{The $p$-elementary divisors}\label{charp}
\subsection{Skew lines, $A$}\hfil \\
By Lemma~\ref{srgeq}, we have
\begin{equation}
A^{2} = kI + \lambda A + \mu(J - A - I).
\end{equation}

Rewriting, we get 
\begin{equation}
\label{eq:matrix1}
A [A - (\lambda - \mu)I] = (k-\mu)I + \mu J,
\end{equation}
where one checks that
\begin{align*}
\lambda - \mu &= -q^{n-2} - q^{n-3} - \cdots - q^{2} + q \\
&= -q(q^{n-3} + q^{n-4} + \cdots +q - 1) 
\end{align*}
and
\begin{align*}
k - \mu &= q^{n-1} + q^{n-2} + \cdots + q^{3} \\
&= q^{3}(q^{n-4} + q^{n-5} + \cdots + 1).
\end{align*}
We can use Equation~\ref{eq:matrix1} to learn much about the $p$-elementary divisors of $A$.  We may work over $\Zp$, the ring of $p$-adic integers.  Let $\LL_{2}$ denote our vertex set of $2$-spaces.  The matrix $A$ defines a $\Zp$-module homomorphism:
\[
A \colon M \to M,
\]
where $M = \ZpL{2}$ consists of all formal $\Zp$-linear combinations of the elements of $\LL_{2}$, and the map $A$ sends a $2$-space to the formal sum of the $2$-spaces that intersect it trivially.  Since $| \LL_{2} |$ is a unit in $\Zp$, we have 
\[
M = \Zp\allone \oplus Y,
\]
where $Y$ is the kernel of the augmentation map $\epsilon \colon M \to \Zp$ sending all $2$-spaces to $1$, and $\allone$ is the formal sum of all of the $2$-spaces.  The map $A$ respects this decomposition. Note that 
\begin{equation}
A(\allone) = k\allone = q^{4} \qbinom{n-2}{2},
\end{equation}
so $p^{4t}$ occurs at least once as an elementary divisor of $A$.  We will see shortly that it occurs exactly once.  Now we focus our attention on $Y$ and the restriction map  $A|_{Y} \colon Y \to Y$.

The Equation~\ref{eq:matrix1} now becomes
\begin{equation}\label{eq:snfbase}
A|_{Y} [A|_{Y} - qz_{1}I] = q^{3}z_{2}I,
\end{equation}
with $z_{1}, z_{2}$ being units in $\Zp$.  It follows immediately that $A|_{Y}$ has no elementary divisors $p^{i}$ for $i>3t$ (and so $p^{4t}$ does indeed have multiplicity $1$ as an elementary divisor of $A$).  It also follows from Equation~\ref{eq:snfbase} that the multiplicity of $p^{i}$ as an elementary divisor of $A|_{Y}$ is equal to the multiplicity of $p^{3t-i}$ as an elementary divisor of $A|_{Y} - qz_{1}I$.  Since $A|_{Y}$ and $A|_{Y} - qz_{1}I$ are congruent modulo $p^{t}$, they have the same multiplicities for elementary divisors $p^{i}$ with $0 \leq i<t$.  Hence, $p^{i}$ and $p^{3t -i}$ occur as elementary divisors for $A$ with the same multiplicity, for $0 \leq i < t$.

Now we use the eigenspaces of $A$ to determine exactly which $p$-elementary divisor multiplicities are nonzero.  From our argument above we again only need consider $A|_{Y}$ and those $p^{i}$ where $i \leq 3t$.  It will turn out that only those elementary divisors $p^{i}$ satisfying $0 \leq i \leq t$ or $2t \leq i \leq 3t$ will have a nonzero multiplicity.  Write $\Qp$ for the field of quotients of $\Zp$.  Denote by $V_{-q^{2}}$ the $\Qp$-eigenspace for $A$ for the eigenvalue $-q^{2} \qbinom{n-3}{1}$, and denote by $V_{q}$ the eigenspace for $q$.  We have the inclusions
\begin{align*}
V_{-q^{2}} \cap M &\subseteq M_{2t}(A)\\
V_{q} \cap M &\subseteq N_{t}(A).
\end{align*}
Recall that $e_{i}$ denotes the multiplicity of $p^{i}$ as an elementary divisor of $A$.  We can apply Lemma~\ref{eigenspacebound} to get the inequalities
\begin{align}\label{ineq:1}
\qbinom{n}{1} - 1 &\leq e_{2t} + \cdots + e_{3t} \\
\qbinom{n}{2} - \qbinom{n}{1} &\leq e_{0} + \cdots + e_{t}.
\end{align}
Since the left sides of the inequalities~\ref{ineq:1} sum to the dimension of $Y$, both must in fact be equalities.  Thus we see that, aside from one occurrence of the elementary divisor $p^{4t}$, the only other $p$-elementary divisors of $A$ that occur are those $p^{i}$ satisfying $0 \leq i \leq t$ or $2t \leq i \leq 3t$.  Since we have equality throughout in \ref{ineq:1}, we see that we only need to determine the multiplicities of the elementary divisors $p^{i}$ where $0 \leq i < t$ in order to know them all.

In general, if $A_{r,s}$ denotes the zero-intersection incidence matrix between $r$-spaces and $s$-spaces, then we have that 
\[
-A_{r,s} \equiv A_{r,1}A_{1,s} \pmod{p^{t}}  
\]
and so the multiplicities of $p^{i}$ as elementary divisors agree for $A_{r,s}$ and $A_{r,1}A_{1,s}$ when $0 \leq i < t$.  In \cite{skew} the $p$-elementary divisors of the matrix $A_{r,1}A_{1,s}$ were computed.  Our matrix $A$ is just $A_{2,2}$, so we want to look at $A_{2,1}A_{1,2}$.  Before we can state the theorem, we need some notation.

Let $[n-1]^{t}$ denote the set of $t$-tuples with entries from $\{1, 2, \ldots , n-1\}$.  For nonnegative integers $\alpha, \beta$, define the following subsets of $[n-1]^{t}$:
\[
\HH_{\alpha} = \left \{ (s_{0}, \ldots, s_{t-1}) \in [n-1]^{t} \vert \sum_{i=0}^{t-1} \max\{0, 2-s_{i}\} = \alpha \right \}
\]
and
\[
\leftsub{\beta}{\HH} = \left \{ (n - s_{0}, \ldots, n - s_{t-1}) \vert (s_{0}, \ldots , s_{t-1}) \in \HH_{\beta} \right \}.
\]
To each tuple $\vec{s} \in [n-1]^{t}$ we will associate a number $d(\vec{s})$ as follows.  First define the integer tuple $\vec{\lambda} = (\lambda_{1}, \ldots, \lambda_{t-1})$ by 
\[
\lambda_{i} = ps_{i+1} - s_{i}
\]
where the subscripts are read modulo $t$.  For an integer $k$, set $d_{k}$ to be the coefficient of $x^{k}$ in the expansion of $(1 + x + \cdots + x^{p-1})^{n}$, or, explicitly,
\[
d_{k} = \sum_{j=0}^{\lfloor k/p \rfloor} (-1)^{j} \binom{n}{j} \binom{n + k - jp - 1}{n-1}.
\]
Finally, we set $d(\vec{s}) = \prod_{i=0}^{t-1} d_{\lambda_{i}}$.

\begin{theorem}[\cite{skew}, Theorem 5.1 and Remark following its proof]
\label{thm:C}
Let $e_{i}$  denote the multiplicity of $p^{i}$ as a $p$-adic elementary divisor of $A_{2,1}A_{1,2}$.%
%
%
\begin{enumerate}
\item $e_{4t} = 1$.
\item For $i \neq 4t$,
\[
e_{i} = \sum_{\vec{s} \in \Gamma(i)} d(\vec{s}),
\]
where
\[
\Gamma(i) = \bigcup_{\substack{\alpha + \beta = i \\ 0 \le \alpha \le t \\ 0 \le \beta \le t}}  \leftsub{\beta}{\HH} \cap \HH_{\alpha}.
\]
\end{enumerate}
Summation over an empty set is interpreted to result in $0$.
\end{theorem}

\subsection{Skew lines, $L$} \hfil \\
The Laplacian matrix $L$ of the skew-lines graph is defined as 
\[
L = kI - A = q^{4}\qbinom{n-2}{2}I - A.
\]
Thus $L \equiv -A \pmod{p^{4t}}$, and so the $p$-elementary divisor multiplicities are the same for $A$ and $L$, for prime powers $p^{i}$ with $0 \leq i < 4t$.  This actually accounts for all of the $p$-elementary divisors of $L$; this can be read off from the order of the critical group which we now compute.

Since the skew-lines graph is regular, the spectrum of $L$ easily follows from that of $A$.  For $L$ we have eigenvalues 
\[
0, \, q^{4}\qbinom{n-2}{2} + q^{2}\qbinom{n-3}{1}, \, q^{4}\qbinom{n-2}{2} - q
\]
with respective multiplicities $1, \qbinom{n}{1}-1, \qbinom{n}{2}-\qbinom{n}{1}$.  By the Matrix Tree Theorem, the order of the critical group is the product of the nonzero eigenvalues of $L$ divided by the number of vertices.  We have
\begin{equation*}
|K(\Gamma)| = \frac{\left(q^{4}\qbinom{n-2}{2} + q^{2}\qbinom{n-3}{1}\right)^{\qbinom{n}{1}-1} \cdot \left(q^{4}\qbinom{n-2}{2} - q\right)^{\qbinom{n}{2}-\qbinom{n}{1}}}{\qbinom{n}{2}},
\end{equation*}
where $K(\Gamma)$ denotes the critical group of $\Gamma$.  Let $S(\Gamma)$ denote the Smith group of $\Gamma$, and compare the above order with
\begin{equation*}
|S(\Gamma)| = q^{4}\qbinom{n-2}{2} \cdot \left( -q^{2}\qbinom{n-3}{1}\right)^{\qbinom{n}{1}-1} \cdot \left( q^{4}\qbinom{n-2}{2} - q\right)^{\qbinom{n}{2} - \qbinom{n}{1}}.
\end{equation*}
Looking carefully, one sees that the $p$-parts of the two orders differ only by an extra factor of $p^{4t}$ in the order of the Smith group of $\Gamma$.  This is accounted for by the image of $\allone$ under $A$, whereas $L(\allone) = 0$.  In other words, the $p$-elementary divisors of $L$ and $A$ agree on $Y$ and these account for all of the $p$-elementary divisors of $L$.

\subsection{Grassmann, $A^{\prime}$} \hfil \\
The spectrum of $A^{\prime}$ is
\[
 q(q+1)\qbinom{n-2}{1}, \, q^{2}\qbinom{n-3}{1} - 1, \, -q-1
\]
with respective multiplicities $1, \qbinom{n}{1}-1, \qbinom{n}{2}-\qbinom{n}{1}$.  Thus we see that there is only one $p$-elementary divisor of $A^{\prime}$, and that is determined by the image of $\allone$ under $A^{\prime}$.  This image is exactly divisible by $p^{t}$, and so the $p$-Sylow subgroup of $S(A^{\prime})$ is cyclic of order $p^{t}$.

\subsection{Grassmann, $L^{\prime}$} \hfil \\
The spectrum of $L^{\prime}$ is
\[
0, \, -q^{2}\qbinom{n-3}{1} + 1 + q(q+1)\qbinom{n-2}{1}, \, q+1 + q(q+1)\qbinom{n-2}{1}
\]
with respective multiplicities $1, \qbinom{n}{1}-1, \qbinom{n}{2}-\qbinom{n}{1}$.  By the Matrix Tree Theorem the order of the critical group of $\Gamma^{\prime}$ depends only on the nonzero eigenvalues.  Since none of these are divisible by $p$, we see that $L^{\prime}$ has no $p$-elementary divisors.


\section{The $\ell$-elementary divisors for $\ell\neq p$}\label{crosschar}

Let $\ell$ be a prime different from $p$.
Let $\Rell$ be the integral extension 
of the $\ell$-adic integers obtained by adding a primitive $p$-th root of unity
and let $\Qell$ be its field of quotients.
The prime $\ell$ is unramified in $\Rell$ and we set $\Fell=\Rell/\ell\Rell$.
Let ${\LL_i}$ be the set of $i$-dimensional subspaces of $V$
and let $M=\Rell^{\LL_2}$ be the free $\Rell$-module
with basis $\LL_2$, which has the structure of an $\Rell\GL(n,q)$-permutation module.
Let $Y$ be the kernel of the homomorphism $M\to\Rell$ sending each
$2$-space to $1$, namely the map given by the matrix $J$ in the standard
basis of $M$. As before, we shall identify the matrices $A$, $L$, $A'$ and $L'$ 
with the endomorphisms of $M$ that they
define with respect to the standard basis.

We have an $A$-invariant decomposition
$$
M=\Rell\allone\oplus Y
$$
if and only if $\ell\nmid\qbinom{n}{2}$.

On extending scalars to $\Qell$, we have
\begin{equation}
M_{\Qell}=\Qell \oplus\tilde S_1 \oplus\tilde S_2,
\end{equation}
where  $\tilde S_1$ and $\tilde S_2$ are simple $\Qell\GL(n,q)$-modules
(Specht modules over $\Qell$). The matrices $A$, $L$,
$A'$ and $L'$ induce endomorphisms of $M_{\Qell}$, and the 
decomposition above is the decomposition in to eigenspaces of these
endomorphisms. The submodule $\tilde S_1$ affords the ``$r$-eigenvalues''
$r_A$, $r_L$, $r_{A'}$ and $r_{L'}$, while $\tilde S_2$ affords the
``$s$-eigenvalues'' $s_A$, $s_L$, $s_{A'}$ and $s_{L'}$.
We fix the notation $f=\dim_{\Qell}\tilde S_1=\qbinom{n}{1}-1$ and
$g=\dim_{\Qell}\tilde S_2=\qbinom{n}{2}-\qbinom{n}{1}$.

In this section we shall apply the results of \S\ref{genlemmas} to the local $PID$ $(\Rell,(\ell))$  and the  $\Rell\GL(n,q)$-modules $M$ and $Y$, to
to determine the $\ell$-elementary divisors of the matrices $A$, $L$, $A'$ and $L'$.
We shall also need to consider the 
$\Rell\GL(n,q)$-permutation modules $\Rell^{\LL_i}$ on the set of $i$-dimensional 
subspaces of $V$, for $i=0$, $1$, $2$ and $3$, as well as the corresponding
permutation modules over $\Fell$ and $\Qell$. 
These modules have been studied in depth by James \cite{James}, to which
we shall refer for the definition and basic properties of the important class
of related modules called {\it Specht modules}, which are analogous to the better known
Specht modules for the symmetric group.
Lemma~\ref{decompmap} will be applied to $M$ or $Y$, with $G=\GL(n,q)$,
$E=\tilde S_i$, for $i=1$ or $2$, and $X$ equal to the corresponding
Specht module over $R$. In this case, $X/\pi X\cong S_i$, the Specht module over
$\Fell$.  By the theory of Specht modules, we also have  $f=\dim_{\Fell}S_1$ and
$g=\dim_{\Fell}S_2$.  

James \cite[Theorem 13.3]{James} has shown that 
$\Fell^{\LL_i}$ has a a descending filtration
whose subquotients are isomorphic to the Specht modules $S_j$, for $0\leq j\leq i$,
each with multiplicity 1. Moreover, the only simple modules
that can be composition factors of the $S_j$ are certain simple modules 
$D_k$, $k\leq j$. We assume $n\geq 4$ to avoid trivialities. Then the modules
$D_k$,  $k=0$, $1$, $2$ are nonisomorphic. 
We have $D_0=S_0\cong\Fell$, the trivial module. The structure of $S_1$
is also easy to describe; if $\ell\nmid\qbinom{n}{1}$ then  $S_1=D_1$, while
if  $\ell\mid\qbinom{n}{1}$, then $S_1$ has a one-dimensional trivial radical.
In the latter case, one easily sees by self-duality of $\Fell^{\LL_1}$ that
$\Fell^{\LL_1}$ is uniserial with socle series $\Fell$, $D_1$, $\Fell$.
The structure of $S_2$ varies according to arithmetic properties;
its composition factors can be read off from James's results \cite{James}. 
It is also possible to describe all the possible submodule lattices, but we shall only give details in the cases we need
for computing the elementary divisors. A general fact is that
$S_2$ has a unique maximal submodule with quotient $D_2$. 
A consequence of this fact that we shall use later on is that
unless $S_2=D_2$,  $\Fell^{\LL_2}$ has no submodule isomorphic to $D_2$.

The composition multiplicities
$[S_i:D_j]$ are given by \cite[Theorem 20.7]{James};
they are no greater than one, with $[S_i:D_i]=1$.
(In James's notation the module $\Fell^{\LL_i}$ is denoted  
$M_{(n-i,i)}$, the Specht module $S_j$ is denoted $S_{(n-j,j)}$ and the simple module
$D_j$ is denoted $D_{(n-j,j)}$.)
The module $D_0$ is the trivial module and the module $D_1$ 
is the unique nontrivial composition factor of $\Fell^{\LL_1}$ (and of $S_1$).
The module $D_2$ is the unique composition factor of $\Fell^{\LL_2}$ 
(and of $S_2$) that is not isomorphic to $D_0$ or $D_1$.

We summarize the consequences of \cite[Theorem 20.7]{James} that we need, using our notation.
\begin{lemma}\label{Spechtmult} We have the following composition multiplicities.
\begin{enumerate}
\item[(a)] $[S_1:D_0]=1$ if $\ell\mid\qbinom{n}{1}$ and is zero otherwise.  
\item[(b)] $[S_2:D_1]=1$ if $\ell\mid\qbinom{n-2}{1}$ and is zero otherwise.
\item[(c)] $[S_2:D_0]=1$ if one of the following conditions holds 
and is zero otherwise.
\begin{enumerate}
\item[(i)] $\ell=2$ and $n\equiv 1$ or $2\pmod 4$.
\item[(ii)] $\ell\neq 2$, $\ell\mid q+1$ and $\ell\mid \lfloor\frac{n-1}{2}\rfloor$.
\item[(iii)] $\ell\neq 2$, $\ell\mid q-1$ and $\ell\mid n-1$.
\item[(iv)] $\ell\neq 2$, $\ell\nmid q-1$ and $\ell\mid\qbinom{n-1}{1}$.
\end{enumerate}
\end{enumerate}
\qed
\end{lemma}
Here $\lfloor x \rfloor$ denotes the greatest integer that is no greater than $x$.

We end this section with some elementary properties of $q$-binomial
coefficients that will be used frequently without further mention.
\begin{lemma}\label{qbinoms} Suppose $\ell\mid q+1$. Then
\begin{enumerate}
\item[(i)] $\ell\mid \qbinom{m}{1}$ if and only if $m$ is even.
\item[(ii)] Assume that $m$ is even. Then
\begin{equation*}
\qbinom{m}{1}=(q+1)\qqbinom{m/2}{1}:=(q+1)(q^{m-2}+q^{m-4}+\cdots+q^2+1).
\end{equation*}
\end{enumerate}
\end{lemma}
\qed


\section{Grassmann, $L^{\prime}$}\label{Lprime}
\subsubsection{Parameters}
\begin{equation}
\begin{aligned}
k'&=q(q+1)\qbinom{n-2}{1},\quad  \mu'=(q+1)^2, \quad \lambda'=\qbinom{n-1}{1}+q^2-2,\\
r_{L'}&=q(q+1)\qbinom{n-2}{1}-q^2\qbinom{n-3}{1}+1=\qbinom{n}{1},\\
s_{L'}&=q(q+1)\qbinom{n-2}{1}+(q+1)=(q+1)\qbinom{n-1}{1}.
\end{aligned}
\end{equation}

In this section we shall write $r$ for $r_{L'}$ and $s$ for $s_{L'}$.
The multiplicity of $r$ is  $f=\dim_{\Qell}\tilde S_1=\qbinom{n}{1}-1$ and
that of $s$ is $g=\dim_{\Qell}\tilde S_2=\qbinom{n}{2}-\qbinom{n}{1}$.

\begin{equation}
\begin{aligned}
\abs{K(\Gamma')}&= \frac{\qbinom{n}{1}^f((q+1)\qbinom{n-1}{1})^g}{\qbinom{n}{2}}\\
&=\qbinom{n}{1}^{f-1}(q+1)^{g+1}\qbinom{n-1}{1}^{g-1}.
\end{aligned}
\end{equation}

\subsubsection{$\ell\nmid q+1$}
If  $\ell\nmid\qbinom{n}{1}$ and  $\ell\nmid\qbinom{n-1}{1}$.
then, $\ell\nmid \abs{K(\Gamma')}$, so we can assume 
 that $\ell$ divides exactly one of  $\qbinom{n}{1}$ and  $\qbinom{n-1}{1}$, since these numbers are coprime.

Case (i). $\ell\mid \qbinom{n}{1}$  and $\ell\nmid \qbinom{n-1}{1}$. 
Let $a=v_\ell(\qbinom{n}{1})$.  
Then  $v_\ell(\abs{K(\Gamma')})=a(f-1)$.
By Lemma~\ref{eigenspacebound} and the definitions of $M_i$ and $e_i$, we have
\begin{equation}
f\leq\dim_\Fell\overline M_a.
\end{equation}
Therefore, by Lemma~\ref{eldiv} with $d=a(f-1)$, $h=1$, $b_2=1$, $b_1=f$ and $a_1=a$,
we obtain
\begin{equation}
e_a=f-1, \quad e_0=g+1, \quad\text{and $e_i=0$ otherwise.} 
\end{equation}

(Moreover, from Lemma~\ref{Spechtmult}, $\dim_{\Fell}D_1=f-1$ and in fact it is
not hard to deduce that $\overline M_a/\overline M_{a+1}\cong D_1$.)

Case (ii). $\ell\nmid \qbinom{n}{1}$  and $\ell\mid \qbinom{n-1}{1}$.
Let $a=v_\ell(\qbinom{n-1}{1})$. 
Then  $v_\ell(\abs{K(\Gamma')})=a(g-1)$.
By Lemma~\ref{eigenspacebound}
\begin{equation}
g\leq \dim_\Fell\overline M_a.
\end{equation}
Therefore, by Lemma~\ref{eldiv} with $d=a(g-1)$, $h=1$, $b_2=1$, $b_1=g$ and $a_1=a$,
we obtain
\begin{equation}
e_a=g-1, \quad e_0=f+1, \quad\text{and $e_i=0$ otherwise.} 
\end{equation}
(Moreover, from Lemma~\ref{Spechtmult}, $\dim_{\Fell}D_2=g-1$ and in fact it is
not hard to deduce that
$\overline M_a/\overline M_{a+1}\cong D_2$.)

\subsubsection{$\ell\mid q+1$}

Case (i) $\ell\nmid\qbinom{n}{1}$, $\ell\mid\qbinom{n-1}{1}$, $\ell\nmid \qqbinom{(n-1)/2}{1}$. Let $a=v_\ell(q+1)$.
Then since  $\qbinom{n-1}{1}=(q+1)\qqbinom{(n-1)/2}{1}$, we have
$v_\ell(\abs{K(\Gamma')})=a(g+1)+a(g-1)=2ag$ and $v_\ell(s)=2a$. 
Also, $\ell\nmid\qbinom{n}{2}$, so $M=\Rell\allone\oplus Y$, where 
$Y$ is the kernel of the map $J$ sending each $2$-subspace
to $\allone$, and $L'$ maps each summand to itself. We have $\Rell\allone=\Ker(L')$
and $Y_\Qell$ contains the $s$-eigenspace of dimension $g$. 
Hence by Lemma~\ref{eigenspacebound}, applied to $Y$, we have $g\leq \dim_\Fell\overline{M_{2a}\cap Y}$, so
\begin{equation}
g+1\leq \dim_\Fell\overline{M}_{2a}.
\end{equation}

Therefore, by Lemma~\ref{eldiv} with $d=2ag$, $h=1$, $b_2=1$, $b_1=g+1$ and $a_1=2a$,
we obtain
\begin{equation}
e_{2a}=g, \quad e_0=f, \quad\text{and $e_i=0$ otherwise.} 
\end{equation}
(Moreover, from Lemma~\ref{Spechtmult}, $\dim_{\Fell}D_2=g$, $\dim_{\Fell}D_1=f$ and in fact it is
not hard to deduce that $\overline M\cong\Fell\allone\oplus D_1\oplus D_2$, with
$\overline M_{2a}/\overline M_{2a+1}\cong D_2$.)

Case (ii) $\ell\nmid\qbinom{n}{1}$, $\ell\mid\qbinom{n-1}{1}$, $\ell\mid \qqbinom{(n-1)/2}{1}$.
Let  $a=v_\ell(q+1)$ and $b=v_\ell(\qqbinom{(n-1)/2}{1})$.
Then   $v_\ell(\abs{K(\Gamma')})=a(g+1)+(a+b)(g-1)=(2a+b)(g-1)+2a$ and $v_\ell(s)=2a+b$. 

We will extract some information from the equation (Lemma~\ref{srgeq})
\begin{equation}\label{Lprimemu}
(L'-sI)(L'-rI)=\mu'J.
\end{equation}
Modulo $\ell$ we have
\begin{equation}
(\overline{L'})(\overline{L'}- rI)=0.
\end{equation}
Since $r\neq 0\pmod\ell$ this equation 
shows that $\overline M$ has a direct sum decomposition into 
eigenspaces $\Ker\overline{L'}$  and $\Ker(\overline{L'}-rI)$.
Since $s\equiv 0\pmod\ell$, the (algebraic) multiplicity of $0$ as an eigenvalue
of $\overline {L'}$ is $g+1$, whence $\dim_\Fell\Ker\overline{L'}=g+1$ and
$\dim_\Fell\Ker(\overline{L'}-rI)=f$.
It also follows that $\Ker\overline{L'}=\Image(\overline{L'}-rI)$, so we have shown
that $\dim_\Fell\overline{\Image(L'-rI)}= g+1$.
Now (\ref{Lprimemu}) implies that for any $x\in\Image(L'-rI)$ we have
$L'x\in sx+\mu'\Rell\allone$, so, as $\ell^{2a}\mid s$ and 
$\ell^{2a}\mid \mu'$ we have $\Image(L'-rI)\subseteq M_{2a}$.
We have proved that 
\begin{equation}
g+1\leq \dim_\Fell\overline M_{2a}.
\end{equation}
By Lemma~\ref{eigenspacebound}  we also have 
\begin{equation}
g\leq \dim_\Fell\overline M_{2a+b}.
\end{equation}
Therefore, by Lemma~\ref{eldiv} with $d=(2a+b)(g-1)+2a$, $h=2$, $b_3=1$, $b_2=g$, $b_1=g+1$,  $a_2=2a+b$ and  $a_1=2a$, 
we obtain
\begin{equation}
e_{2a+b}=g-1,\quad  e_{2a}=1 \quad e_0=f, \quad\text{and $e_i=0$ otherwise.} 
\end{equation}

For the next two cases, we need the following lemma.
\begin{lemma}\label{ptlines} Assume that $\ell\mid q+1$ and that $\ell\mid\qbinom{n}{1}$. 
Let $a=v_\ell(q+1)$.
Let $W_P$ be the set of $2-spaces$
containing a fixed $1$-space $P$, and let $[W_P]\in M$ be its characteristic
vector. Then $ [W_P]\in M_a$. 
\end{lemma}
\begin{proof}
Since $a=v_\ell(q+1)$ and  $q+1\mid k'$, it suffices to show that $[W_P]\in M_a(A')$.
Let $A'([W_P])=\sum_{m\in\LL_2} a_m m$. Then the coefficient $a_m$
is the number of lines through $P$ that are distinct from $m$ and meet $m$.
This number is $\qbinom{n-1}{1}-1=q\qbinom{n-2}{1}$ if $m$ passes through $P$ and 
$q+1$ if $m$ does not. Since $q+1\mid\qbinom{n}{1}$ under our assumptions, we also have $q+1\mid\qbinom{n-2}{1}$, so the Lemma holds.
\end{proof}

Case (iii) $\ell\mid\qbinom{n}{1}$, $\ell\nmid\qbinom{n-1}{1}$, 
$\ell\nmid\qqbinom{n/2}{1}$,  $\ell\nmid\lfloor\frac{n-1}{2}\rfloor$. 
Then $\qbinom{n}{1}=(q+1)\qqbinom{n/2}{1}$. Let $a=v_\ell(q+1)$.
Then   $v_\ell(\abs{K(\Gamma')})=a(f+g)$. 
We have $\ell\nmid\qbinom{n}{2}$, so shall always apply
Lemma~\ref{eigenspacebound} to the complementary summand $Y$ to $\Rell\allone$ in $M$.
Write
\begin{equation}
(L'-rI)(L'-sI)=\mu'J
\end{equation}
as
\begin{equation}
L'(L'-(r+s)I)=-rsI+\mu'J
\end{equation}

to get $\Image(L'-(r+s)I)\subseteq M_{2a}$, hence
$\overline M_{2a}\supseteq \overline{\Image(L'-(r+s)I)}=
\overline{\Image(L')}=\Image(\overline{L'})$.
It is easy to see that $\Image(\overline{L'})$ is not in the span of $\allone$, and that
$\allone\in M_{2a}$. Since $S_2$ is not simple in this case, by Lemma~\ref{Spechtmult}, the only
simple module that can appear in the socle of $\overline Y$ is $D_1$.
It follows that $\dim_\Fell\overline M_{2a}\geq 1+\dim_\Fell D_1=f$.
Next, we will show that $\dim_{\Fell}(\overline{M}_{a}) \geq g+2$.  Over $\Qell$, the submodule $\tilde S_2$ is the $s$-eigenspace
of $M_\Qell$, so by Lemma~\ref{decompmap} and Lemma~\ref{Spechtmult}
the composition factors of $\overline{M\cap\tilde S_2}$ are $D_1$ and $D_2$.

The map $\eta$ from $\Rell^{\LL_1}\to M$ sending $P$ 
to $[W_P]$ is an $\Rell\GL(n,q)$-homomorphism. From the structure
of $\Fell^{\LL_1}$, it is not hard to see that 
$\Image(\overline\eta)$ has composition factors 
$\Fell$ and $D_1$, and simple socle $D_1$.  In particular, $\Image(\overline\eta)$ does not contain $\allone$.  Using the preceding paragraph and Lemma~\ref{ptlines}  
we now see that  $\overline{M}_{a}$ has among its composition factors $\Fell$ (occurring twice), $D_{1}$ and $D_2$. This proves our claim  that $\dim_{\Fell}(\overline M_a)\geq g+2$. 

Therefore, by Lemma~\ref{eldiv} with $d=a(f+g)$, $h=2$, $b_3=1$, $b_2=f$, $b_1=g+2$,  $a_2=2a$ and  $a_1=a$, 
we obtain
\begin{equation}
e_{2a}=f-1, \quad e_{a}=g-f+2, \quad e_0=f-1, \quad\text{and $e_i=0$ otherwise.} 
\end{equation}

Case(iv)   $\ell\mid\qbinom{n}{1}$, $\ell\nmid\qbinom{n-1}{1}$,
$\ell\mid\qqbinom{n/2}{1}$. In this case $\ell\mid\frac{n}{2}$ so
$\ell\nmid\lfloor\frac{n-1}{2}\rfloor$. Let $a=v_\ell(q+1)$
and $b=v_\ell(\qqbinom{n/2}{1})$.
Then   $v_\ell(\abs{K(\Gamma')})=(a+b)f+ag-b$. 
We claim that  $\dim_{\Fell}(\overline M_a)\geq g+2$.  
Consider the map $\eta:\Rell^{\LL_1}\to M$ sending a $1$-subspace $P$ to $[W_P]$.
Since $\ell\mid q+1$, the all-one vector of $\Fell^{\LL_1}$ is in the kernel of 
$\overline\eta$, and from the uniserial structure of $\Fell^{\LL_1}$ it is not hard to see that its span equals the kernel.
It follows that $\Image(\overline\eta)$ is of dimension $f$ and has no fixed points,
and, in particular, does not contain $\allone$. However, by Lemma~\ref{ptlines} 
$\Image(\eta)\subseteq M_a$, so we know that
$\overline M_a$ has composition factors $F$ (twice), $D_1$.
Also, since $M\cap\tilde S_2\subseteq M_a$, we know by Lemma~\ref{decompmap} that
$\overline M_a$ contains the composition factors $D_1$ and $D_2$ of $S_2$.
It follows that $\overline M_a$ must have at least composition
factors $F$ (twice), $D_1$ and $D_2$, so is of dimension $\geq g+2$.

Let $Y=\Ker(J)$. Here $Y$ is not a direct summand of $M$ 
as an  $\Rell\GL(n,q)$-module but it is a pure $\Rell\GL(n,q)$-submodule.

\begin{lemma}\label{Ybariv} Under the current hypotheses on $\ell$, we have 
$\Fell\allone\subseteq\overline Y$ and $\overline Y/\Fell\allone$ is uniserial with socle series $D_1$, $D_2$, $D_1$. 
\end{lemma}
\begin{proof}
The composition factors are given by Lemma~\ref{Spechtmult}.
Since $\ell\mid\qbinom{n}{2}$ we have 
$\Fell\allone\subseteq\overline Y=(\Fell\allone)^\perp$, and $\overline Y/\Fell\allone$ is self-dual with composition factors $D_1$, $D_2$, $D_1$. The image of the Specht $S_2$ in
$\overline Y/\Fell\allone$ has composition factors $D_2$ and $D_1$,
so is uniserial. The lemma now follows from the 
self-duality of $\overline Y/\Fell\allone$.
\end{proof}


We claim that  $\dim_{\Fell}(\overline M_{2a+b})\geq f$. To see this consider
the image of  $Y$ under $L'-(r+s)I$.  
Since $Y=\Ker(J)$, the equation
\begin{equation}
L'(L'-(r+s)I)=-rsI+\mu'J
\end{equation}
implies that  $(L'-(r+s)I)(Y)\subseteq M_{2a+b}$.
It follows that $\overline{L'}(\overline Y)\subseteq\overline M_{2a+b}$,
and it is easy to check that $\overline{L'}(\overline Y)$ is not in the span of
$\allone$.
Since also $\allone\in M_{2a+b}$, it follows 
that $\overline M_{2a+b}$ has at least one trivial composition factor
and one composition factor isomorphic to $D_1$,  so its dimension is at least $f$.

Therefore, by Lemma~\ref{eldiv} with $d=(a+b)f+ag-b$, $h=2$, $b_3=1$, $b_2=f$, $b_1=g+2$,  $a_2=2a+b$ and  $a_1=a$, 
we obtain
\begin{equation}
e_{2a+b}=f-1,\quad  e_{a}=g-f+2, \quad e_0=f-1, \quad\text{and $e_i=0$ otherwise.} 
\end{equation}

Case(v)  $\ell\mid\qbinom{n}{1}$, 
$\ell\nmid\qbinom{n-1}{1}$, 
$\ell\nmid\qqbinom{n/2}{1}$, 
$\ell\mid\lfloor\frac{n-1}{2}\rfloor$. Let
$a=v_\ell(q+1)$.
Here,  $r$ and $s$ are exactly divisible by $\ell^a$ and
$\ell\nmid \qbinom{n}{2}$, so $v_\ell(\abs{K(\Gamma')})=a(f+g)$.
Also, since $\ell\nmid \qbinom{n}{2}$, $\overline M=\Fell\allone\oplus \overline Y$, where
$Y=\Ker(J)$. This is an orthogonal decomposition with respect to the 
dot product on $\overline M$.

\begin{lemma}\label{uniserial_module} $\overline Y$ is a uniserial
$\Fell\GL(n,q)$-module with socle series $D_1$, $\Fell$, $D_2$, $\Fell$, $D_1$.
\end{lemma}
\begin{proof}
The composition factors of $\overline Y$ and the Specht 
modules can be read off from Lemma~\ref{Spechtmult}. 
Since the space of $\GL(n,q)$-fixed points on $\overline M$ is
$\Fell\allone$, there are no fixed points on $\overline Y$, hence none
on its submodule $S_2$. Since $D_2$ is the unique simple quotient
of $S_2$, it now follows that $S_2$ is uniserial with socle series
$D_1$, $\Fell$, $D_2$. The lemma is now implied by the self-duality
of $\overline Y$.
\end{proof}

\begin{lemma}\label{planelines} Let $W_\pi$ be the set of $2-spaces$
contained in a fixed $3$-space $\pi$, and let $[W_\pi]\in M$ be its characteristic
vector. Then $\allone-[W_\pi]\in M_a(L')$ and $\allone-[W_\pi]\in M_a(A')$.
\end{lemma}
\begin{proof} Since $q+1\mid k'$, it suffices to prove that $A^{\prime}[W_\pi]\in\ell^aM$.
Write $A^{\prime}[W_\pi]=\sum_L c_LL$, where $L$ runs over the $2$-subspaces.
Then it is easy to see that $c_L=0$ if $L$ has zero intersection with $\pi$,
that $c_L=q+1$ if $L$ meets $\pi$ in a $1$-dimensional space, and that
$c_L=q^2+q$ if $L$ is contained in $\pi$. 
\end{proof}

\begin{lemma}\label{uniserialY}
Let $\langle -,-\rangle$ denote the natural dot product on $\overline M$. 
Let $\pi$ be a $3$-subspace and $P$ a $1$-subspace of $\pi$.
Under the current hypotheses on $\ell$, the following hold, where we denote images
in $\overline M$ of elements of $M$ by the same symbols.
\begin{enumerate}
\item[(a)] $\langle\allone, \allone-[W_P]\rangle=0$.
\item[(b)] $\langle\allone, \allone-[W_\pi]\rangle=0$.
\item[(c)] $\langle\allone-[W_P],\allone-[W_\pi]\rangle\neq 0$.
\end{enumerate}
\end{lemma}
\begin{proof} Under the hypothesis on $\ell$, we know that $n$ is even, that
$\qbinom{n}{2}\equiv\frac{n}{2}\equiv 1\pmod\ell$, as $\ell\mid(\frac n2-1)$.
The number of $2$-spaces containing $P$ is 
$\qbinom{n-1}{1}\equiv 1\pmod\ell$. 
The number of $2$-spaces contained in $\pi$ is 
$q^2+q+1\equiv 1\pmod\ell$, and of these there are $q+1$
which contain $P$. All parts of the lemma now follow.
\end{proof}
We claim that  $\dim_{\Fell}(\overline M_a)\geq g+2$. 
To see this we make use of the uniserial structure of $\overline Y$ 
given in Lemma~\ref{uniserial_module}. Since $\ell\mid n$, the permutation module
on points $\Fell^{\LL_1}$ is uniserial with series $\Fell$, $D_1$, $\Fell$.
Let $U_1$ be the submodule of $\overline M$ generated by the elements 
$\allone-[W_Q]$ as $Q$  ranges over $\LL_1$. Then $U_1\subseteq \overline Y$
by Lemma~\ref{uniserialY}(a). Also, $U_1$ is the image of 
$\Fell^{\LL_1}$ under the $\Fell\GL(n,q)$-map sending $Q$ to $\allone-[W_Q]$, so
$U_1$ is uniserial, and from the structures of $\overline Y$ and $\Fell^{\LL_1}$,
we see that $U_1$ has socle series $D_1$, $\Fell$, so is equal to the second socle,
$\soc^2(\overline Y)$.
Thus, from the structure of $\overline Y$, we see that $U_1^\perp\cap\overline Y$
is equal to $\soc^3(\overline Y)$.
Let $U_3$ be the submodule of $\overline M$ generated by the elements
$\allone-[W_\pi]$ as $\pi$ ranges over $\LL_3$. Then $U_3\subseteq\overline Y$ 
by Lemma~\ref{uniserialY}(b), but $U_3\nsubseteq U_1^\perp\cap\overline Y$
by Lemma~\ref{uniserialY}(c). Therefore, from the uniserial structure of
$\overline Y$, it follows that $U_3\supseteq\soc^4(\overline Y)$, so
contains at least the composition factors $D_1$, $F$ (twice) and $D_2$.
Hence $\dim_{\Fell}U_3\geq g+1$. Our claim now follows from 
Lemma~\ref{planelines} and the fact that $\allone\in\Ker(L')$.

We next claim that  $\dim_{\Fell}(\overline M_{2a})\geq f$.  
From the equation 
\begin{equation}
L'(L'-(r+s)I)=-rsI+\mu'J
\end{equation}
we have $\Image(L'-(r+s)I)\subseteq M_{2a}$, so
$\Image(\overline{L'})\subseteq \overline M_{2a}$.
This image is easily seen not to be in the span of $\allone$, hence
from the structure of $\overline M$, must have a composition factor isomorphic
to $D_1$. Also $\allone\in M_{2a}$, so $\overline M_{2a}$ has dimension
at least $1+\dim_{\Fell} D_1=f$. 

Therefore, by Lemma~\ref{eldiv} with $d=a(f+g)$, $h=2$, $b_3=1$, $b_2=f$, $b_1=g+2$,  $a_2=2a$ and  $a_1=a$, 
we obtain
\begin{equation}
e_{2a}=f-1,\quad  e_{a}=g-f+2, \quad e_0=f-1, \quad\text{and $e_i=0$ otherwise.} 
\end{equation}


\section{Grassmann, $A^{\prime}$} 
\subsubsection{Parameters}
\begin{equation}
\begin{aligned}
k'&=q(q+1)\qbinom{n-2}{1},\quad  \mu'=(q+1)^2, \quad \lambda'=\qbinom{n-1}{1}+q^2-2,\\
r_{A^{\prime}}&=q^2\qbinom{n-3}{1}-1=q\qbinom{n-2}{1}-(q+1), \quad s_{A^{\prime}}=-(q+1).
\end{aligned}
\end{equation}

In this section we shall write $r$ for $r_{A'}$ and $s$ for $s_{A'}$. 
The multiplicity of $r$ is  $f=\dim_{\Qell}\tilde S_1=\qbinom{n}{1}-1$ and
that of $s$ is $g=\dim_{\Qell}\tilde S_2=\qbinom{n}{2}-\qbinom{n}{1}$.  

\begin{equation}
\abs{S(\Gamma')}=q(q+1)^{g+1}(q^{2}\qbinom{n-3}{1}-1)^f\qbinom{n-2}{1}.
\end{equation}

\begin{lemma}\label{NumThyAprime}
\begin{enumerate}
\item[(a)] If $n$ is odd then $\gcd(r,s)=1$, $\gcd(k',s)=q+1$ and $\gcd(k',r)=1$.
\item[(b)]If $n$ is even, then $\gcd(r,s)=q+1$, $\gcd(s,k')=q+1$,  and any common prime divisor of $r$ and $k'$ must divide $q+1$.
\end{enumerate}
\end{lemma}

\subsubsection{$\ell\nmid q+1$}

In this case, $\ell\nmid s$. In the nontrivial case where $\ell\mid k'r$, we know by
Lemma~\ref{NumThyAprime} that $\ell$ divides exactly one of $k'$ and $r$.

Case(i). Suppose $v_\ell(k')=a$, $a>0$, $\ell\nmid r$, $\ell\nmid s$.
Here, $v_\ell(\abs{S(\Gamma')})=a$ and we have $\allone\in M_a$, so
$\dim_{\Fell}\overline M_a\geq 1$.
Therefore, by Lemma~\ref{eldiv} with $d=a$, $h=1$, $b_2=0$, $b_1=1$ and  $a_1=a$, 
we obtain
\begin{equation}
e_{a}=1, \quad e_0=f+g, \quad\text{and $e_i=0$ otherwise.} 
\end{equation}

Case(ii). Suppose $v_\ell(r)=a$, $a>0$, $\ell\nmid k'$, $\ell\nmid s$.
Here, $v_\ell(\abs{S(\Gamma')})=af$ and by Lemma~\ref{eigenspacebound} we
have $\dim_{\Fell}\overline M_a\geq f$. 
Therefore, by Lemma~\ref{eldiv} with $d=af$, $h=1$, $b_2=0$, $b_1=f$ and  $a_1=a$, 
we obtain
\begin{equation}
e_{a}=f, \quad e_0=g+1, \quad\text{and $e_i=0$ otherwise.} 
\end{equation}

\subsubsection{$\ell\mid q+1$}
Since $\ell\mid q+1$, we have for any $m$ that $\ell\mid\qbinom{m}{1}$ if and only of
$m$ is even, so we divide into two cases according to the parity of $n$.

i) $n$ even. Then 
\begin{equation}
k'=q(q+1)^2 h, \quad r=(q+1)(qh-1), \quad s=-(q+1),
\end{equation}
where $h=\qqbinom{\frac{n-2}{2}}{1}$. Note that $\ell\mid h$ if and only if 
$\ell\mid\frac{n-2}{2}$ and that $\gcd(h, qh-1)=1$.

Let $a=v_\ell(q+1)$,  $b=v_\ell(qh-1)$ and  $c=v_\ell(h)$.
Then either $c=0$ or $b=0$.

Note that since $\ell\mid q+1$ and $n$ is even, we have
$\ell\mid\qbinom{n}{2}$ iff $\ell\mid\frac{n}{2}=\frac{n-2}{2}+1$.
But modulo $\ell$, we have $-(\frac{n-2}{2}+1)\equiv -h-1\equiv qh-1$,
so we see that $b=0$ if and only if $\ell\nmid\qbinom{n}{2}$.

a) Suppose $c=0$ and $b=0$.
In this case $\ell\nmid\qbinom{n}{2}$, so we have a decomposition
$M=\Rell\allone \oplus Y$. By Lemma~\ref{Spechtmult} 
the module $\overline Y$ has composition factors $D_1$ (twice), $D_2$, $\Fell$.
From the self-duality of $\overline Y$ and the fact that $\GL(n,q)$ has no fixed
points in $\overline Y$, one sees that $\overline Y$ has socle series $D_1$, $\Fell\oplus D_2$,
$D_1$. By Lemma~\ref{srgeq}, on $Y=\Ker(J)$ we have $(A'-rI)(A'-sI)=0$, which we can write as
\begin{equation}
A'(A'-(r+s)I)=-rsI.
\end{equation}
This shows that $(A'-(r+s)I)(Y)\subseteq M_{2a}\cap Y$, so as $\ell\mid(r+s)$,
we have $\overline{A'}(\overline Y)\subseteq\overline{M_{2a}\cap Y}$.
It is easy to see that $\overline{A'}(\overline Y)$ is nonzero, 
so it has $D_1$ as a composition factor. In particular, $\dim_{\Fell}(\overline{M_{2a}\cap Y})\geq f-1$ and so $\dim_{\Fell}\overline M_{2a}\geq f$, as $\allone\in M_{2a}$. 
By Lemma~\ref{decompmap} applied to $Y$ we know from the
fact that $Y_K=\tilde S_1\oplus\tilde S_2$ that 
$\overline{M_{a}\cap Y}$ must contain the composition factors of $S_1$
and of $S_2$, so has at least the composition factors $D_1$, $\Fell$ and $D_2$.
This shows $\dim_{\Fell}(\overline{M_{a}\cap Y})\geq (f-1)+1+(g-f+1)=g+1$, so
$\dim_{\Fell}\overline M_{a}\geq g+2$.
Now $v_\ell(\abs{S(\Gamma')})=a(f+g)+2a$.
Therefore, by Lemma~\ref{eldiv} with $d=a(f+g)+2a$, $h=2$, $b_3=0$, $b_2=f$,
$b_1=g+2$, $a_2=2a$ and  $a_1=a$, 
we obtain
\begin{equation}
e_{2a}=f, \quad e_a=g-f+2,\quad  e_0=f-1, \quad\text{and $e_i=0$ otherwise.} 
\end{equation}

b) Suppose $b>0$, which implies $c=0$.
Here, $v_\ell(\abs{S(\Gamma')})=(a+b)f+ag+2a$.
\begin{lemma} Let $Y=\Ker(J)$. Then $\Fell\allone\subseteq\overline Y$
and $\overline Y/\Fell\allone$ is uniserial with socle series $D_1$, $D_2$, $D_1$.
\end{lemma}
We will prove the following bounds. 
\begin{lemma}
\begin{enumerate}
\item $\dim_{\Fell}\overline M_{2a+b}\geq f$. 
\item $\dim_{\Fell}\overline M_{a}\geq g+2$.
\end{enumerate}
\end{lemma}
\begin{proof} The restriction of $A'$ to $Y$ satisfies the equation
\begin{equation}
A'(A'-(r+s)I)=-rsI.
\end{equation}
This shows that $(A'-(r+s)I)(Y)\subseteq M_{2a+b}$, so
$(\overline{A'-(r+s)I})(\overline Y)\subseteq \overline M_{2a+b}$
It is easy to see that $(\overline{A'-(r+s)I})(\overline Y)$ is not in the span 
of $\allone$ so it has a composition factor $D_1$. 
By explicit computation we have
\begin{equation}\label{ImAp}
A'([W_p])=(q\qbinom{n-2}{1}-(q+1))[W_p]+(q+1)\allone=
(q+1)(qh-1)[W_p]+(q+1)\allone
\end{equation}
and 
\begin{equation}
A'(\allone)=k'\allone=q(q+1)^2h\allone.
\end{equation}
Therefore, as $h$ is a unit of $\Rell$, we have
$(q+1)[W_p]-(qh)^{-1}\allone\in M_{2a+b}$,
 whence $\Fell\allone\subseteq\overline M_{2a+b}$. This proves (1).

By Lemma~\ref{decompmap} we know that $\overline M_a$ has
at least the composition factors $D_2$ and $D_1$ of $S_2$, and we have
already seen that $\Fell\allone\subseteq\overline M_a$.
We shall show that $\overline M_a$ contains a submodule with 
no $\GL(n,q)$-fixed points but with a trivial composition factor.
This will show that $\overline M_a$ at least two trivial composition factors
and (2) will be proved. The submodule in question is the
$\Fell$-submodule $W$ generated by the $[W_p]$, for $1$-spaces $p$.
It is the image of the homomorphism  $\eta:\Fell^{\LL_1}\to M$
sending a $1$-space to the sum of $2$-spaces containing it.
A direct computation shows that the image is not in the span of $\allone$ and that
the kernel contains the all-one vector of $\Fell^{\LL_1}$.
Since $\Fell^{\LL_1}$ is uniserial with socle series $\Fell$, $D_1$, $\Fell$,
it follows that $W$ is uniserial with socle series $D_1$, $\Fell$.
Equation (\ref{ImAp}) shows that $A'[W_p]\in M_a$, so we are done.
\end{proof}

In view of the above lemma can apply  Lemma~\ref{eldiv} with 
$d=(a+b)f+ag+2a$, $h=2$, $b_3=0$, $b_2=f$,
$b_1=g+2$, $a_2=2a+b$ and  $a_1=a$, 
we obtain
\begin{equation}
e_{2a+b}=f, \quad e_a=g-f+2,\quad  e_0=f-1, \quad\text{and $e_i=0$ otherwise.} 
\end{equation}

c) Suppose $c>0$, which implies $b=0$.
Then $\ell\nmid\qbinom{n}{2}$.
We have a decomposition $M=\Rell\allone\oplus Y$. 
Here, $v_\ell(\abs{S(\Gamma')})=af+ag+2a+c$.
We will prove the following bounds. 
\begin{enumerate}
\item $\dim_{\Fell}\overline M_{2a+c}\geq 1$. 
\item $\dim_{\Fell}\overline M_{2a}\geq f$. 
\item $\dim_{\Fell}\overline M_{a}\geq g+2$.
\end{enumerate}
(1) is by simply computing $A'(\allone)$.
For (2) and (3) we note that the current hypotheses
on $\ell$ imply those of case (v) for $L'$ in \S\ref{Lprime}, and that
as $\ell^{2a}\mid k'$ we have  $M_{2a}(A')=M_{2a}(L')$ and $M_{a}(A')=M_{a}(L')$.

Therefore, by Lemma~\ref{eldiv} with $d=af+ag+2a+c$, $h=3$, $b_4=0$, $b_3=1$, $b_2=f$,
$b_1=g+2$, $a_3=2a+c$, $a_2=2a$ and  $a_1=a$, 
we obtain
\begin{equation}
e_{2a+c}=1,\ e_{2a}=f-1, \ e_a=g-f+2,\  e_0=f-1, \ \text{and $e_i=0$ otherwise.} 
\end{equation}

ii) $n$ odd. Then $\ell\nmid\qbinom{n}{1}$ and $\ell\nmid\qbinom{n-2}{1}$,
Let $a=v_\ell(q+1)$. Then, by Lemma~\ref{NumThyAprime},
$a=v_\ell(k')=v_\ell(s)$ and $\ell\nmid r$.
Here, $v_\ell(\abs{S(\Gamma')})=a+ag$.

a) $\ell\nmid\lfloor\frac{n-1}{2}\rfloor$.
In this case we have a decomposition $M=\Rell\allone\oplus Y$.
Applying Lemma~\ref{eigenspacebound} to $Y$, we have
$\dim_{\Fell}\overline{(M_a\cap Y)}\geq g$, and so since $\allone\in M_a$, we have
$\dim_{\Fell}\overline M_a\geq g+1$.

Therefore, by Lemma~\ref{eldiv} with $d=a(g+1)$, $h=1$, $b_2=0$, 
$b_1=g+1$ and  $a_1=a$, 
we obtain
\begin{equation}
e_{a}=g+1,\quad  e_0=f, \quad\text{and $e_i=0$ otherwise.} 
\end{equation}

b) $\ell\mid\lfloor\frac{n-1}{2}\rfloor$.
Let $a=v_\ell(q+1)$.
We claim that $\dim_{\Fell}\overline M_a\geq g+1$.
By Lemma~\ref{srgeq}, we have $(A'-sI)(A'-rI)=\mu' J$, where $\mu'=(q+1)^2$.
Thus, $A'(A'-rI)=s(A'-rI)+\mu' J$, which shows that
$\Image(A'-rI)\subseteq M_a$ and hence $\Image(\overline{A'-rI})\subseteq \overline M_a$.
Now  on $\overline M$, $\overline{A'}$ has a quadratic minimal polynomial 
with distinct roots $0$ and $r$, so we have an $\Fell\GL(n,q)$-decomposition
\begin{equation}
\overline M=\Ker(\overline{A'-rI})\oplus \Ker\overline{A'}
=\Ker(\overline{A'-rI})\oplus \Image(\overline{A'-rI}).
\end{equation}
The dimension of $\Ker\overline{A'}$ is equal to the algebraic
multiplicity of $0$ as an eigenvalue of $\overline{A'}$, which equals
$g+1$ since  $\ell$ divides both $s$ and $k'$. 
Therefore $\dim_{\Fell} \overline M_a\geq g+1$.
It now follows from Lemma~\ref{eldiv}
with $d=a(g+1)$, $h=1$, $b_2=0$, 
$b_1=g+1$ and  $a_1=a$, 
that
\begin{equation}
e_{a}=g+1,\quad  e_0=f, \quad\text{and $e_i=0$ otherwise.} 
\end{equation}


\section{Skew lines, $L$}
\subsubsection{Parameters} 
\begin{equation}
\begin{aligned}
\mu&=\frac{q^3\qbinom{n-3}{1}\left(\qbinom{n-1}{1}-(q+2)\right)}{q+1}\\
r_L&=q^4\qbinom{n-2}{2}-q^2\qbinom{n-3}{1}=\frac{q^2\qbinom{n-3}{1}\qbinom{n}{1}}{q+1},\\
s_L&=q^4\qbinom{n-2}{2}-q=\frac{q\qbinom{n-1}{1}\left(\qbinom{n-1}{1}-(q+2)\right)}{q+1}.
\end{aligned}
\end{equation}

In this section we shall write $r$ for $r_{L}$ and $s$ for $s_{L}$.
The multiplicity of $r$ is  $f=\dim_{\Qell}\tilde S_1=\qbinom{n}{1}-1$ and
that of $s$ is $g=\dim_{\Qell}\tilde S_2=\qbinom{n}{2}-\qbinom{n}{1}$.  

\begin{equation}
\abs{K(\Gamma)}=q^{2f+g}\frac{\left(\frac{\qbinom{n}{1}\qbinom{n-3}{1}}{q+1}\right)^f
\left(\frac{\qbinom{n-1}{1}\left(\qbinom{n-1}{1}-(q+2)\right)}{q+1}\right)^g}{\qbinom{n}{2}}
\end{equation}

\begin{lemma}\label{NumThyL}
\begin{enumerate}
\item[(a)] If $n$ is odd then no prime $\ell\neq p$ can divide both $r$ and $s$.
\item[(b)]If $n$ is even, then any common prime divisor of $r$ and $s$ must
divide $q(q+1)$.
\item[(c)] $\qbinom{n-1}{1}$ divides $s(q+1)$.
\end{enumerate}
\end{lemma}
\begin{proof}

(a) Suppose $\ell\neq p$ divides both $r$ and $s$. Then $\ell$ divides both
$(\qbinom{n-3}{1}/(q+1))\qbinom{n}{1}$ and
$q^3\qbinom{n-2}{2}-1=q^3\qbinom{n-2}{1}(\qbinom{n-3}{1}/(q+1))-1$.
Obviously, $\ell$ cannot divide  $\qbinom{n-3}{1}/(q+1)$, so $\ell$ must divide
$\qbinom{n}{1}=q^2\qbinom{n-2}{1}+(q+1)$. Hence 
$q^3\qbinom{n-2}{1}\equiv -q(q+1) \pmod\ell$ and so
$$
0\equiv q^3\qbinom{n-2}{1}(\qbinom{n-3}{1}/(q+1))-1\equiv -q\qbinom{n-3}{1}-1
\equiv -\qbinom{n-2}{1} \pmod\ell,
$$
a contradiction since $\ell$ divides $q^3\qbinom{n-2}{1}(\qbinom{n-3}{1}/(q+1))-1$.

(b) Suppose that $\ell$ divides $r$ and $s$, but $\ell$ does not divide $q(q+1)$.
Then $\ell$ divides $(\qbinom{n}{1}/(q+1))\qbinom{n-3}{1}$ and
$q^3(\qbinom{n-2}{1}/(q+1))\qbinom{n-3}{1}-1$ so, obviously, $\ell$ does not divide
$\qbinom{n-3}{1}$. Hence $\ell$ divides $\qbinom{n}{1}=q^2\qbinom{n-2}{1}+q+1$.
Thus $q^3\qbinom{n-2}{1}\equiv-q(q+1)\pmod\ell$, and
it follows that $\ell$ divides $-q\qbinom{n-3}{1}-1=-\qbinom{n-2}{1}$.
But then $\ell$ divides $\qbinom{n-2}{1}/(q+1)$, which contradicts the fact that
$\ell$ divides $q^3(\qbinom{n-2}{1}/(q+1))\qbinom{n-3}{1}-1$.

(c) This is clear from the factorization of $s$ given at the beginning of this section.
\end{proof}
\subsubsection{$\ell\nmid q+1$}
By Lemma~\ref{NumThyL}, $\ell$ can divide at most one of $r$ and $s$, so
we consider the two nontrivial cases in turn.

Case(i) $\ell\mid r$, $\ell\nmid s$.

a) $\ell\mid\qbinom{n}{1}$ and  $\ell\nmid\qbinom{n-3}{1}$.
Let  $a=v_\ell(\qbinom{n}{1})$. Then  $a=v_\ell(\qbinom{n}{2})$, and
$v_\ell(\abs{K(\Gamma)})=a(f-1)$. By Lemma~\ref{eigenspacebound},
we have $\dim_{\Fell}(\overline M_a)\geq f$.
Therefore, by Lemma~\ref{eldiv} with $d=a(f-1)$, $h=1$, $b_2=1$, 
$b_1=f$ and  $a_1=a$, 
we obtain
\begin{equation}
e_{a}=f-1,\quad  e_0=g+1, \quad\text{and $e_i=0$ otherwise.} 
\end{equation}

b) $\ell\mid\qbinom{n}{1}$ and  $\ell\mid\qbinom{n-3}{1}$.
Note that we must have $\ell\mid q^2+q+1$.
Let $a=v_\ell(\qbinom{n}{1})$, $b=v_\ell(\qbinom{n-3}{1})$,

Here, $v_\ell(\qbinom{n}{2})=a$, so $v_\ell(\abs{K(\Gamma)})=a(f-1)+bf$.
By Lemma~\ref{eigenspacebound} we have 
$\dim_{\Fell}\overline M_{a+b}\geq f$.
We claim $\dim_{\Fell}\overline M_b\geq f+1$. 

From Lemma~\ref{srgeq} we have
$L(L-(s+r)I)=-rsI+\mu J$, which shows that
$\Image(L-(s+r)I)\subseteq M_b$, hence 
$\Image(\overline{L-sI})\subseteq \overline M_b$.
Now on $\overline M$, $\overline L$ has
quadratic minimal polynomial with eigenvalues $s$ and $0$,
which gives the decomposition of $\Fell\GL(n,q)$-modules
\begin{equation}
\overline M=\Ker(\overline{L})\oplus\Ker(\overline{L-sI})=\Image(\overline{L-sI})\oplus\Ker(\overline{L-sI}).
\end{equation}
The dimension of $\Ker\overline{L}$ is equal to the algebraic
multiplicity of $0$ as an eigenvalue of $\overline{L}$, which equals
$f+1$ since  $\ell$ divides both $r$ and $k$. 
We have shown that $\dim_{\Fell} \overline M_b\geq f+1$.

Therefore, by Lemma~\ref{eldiv} with $d=a(f-1)+bf$, $h=2$, $b_3=1$, $b_2=f$, 
$b_1=f+1$, $a_2=a+b$ and  $a_1=b$, 
we obtain
\begin{equation}
e_{a+b}=f-1,\quad  e_b=1, \quad  e_0=g, \quad\text{and $e_i=0$ otherwise.} 
\end{equation}

c) $\ell\nmid\qbinom{n}{1}$ and $\ell\mid\qbinom{n-3}{1}$.
Under the present assumptions, we have $\ell\nmid\qbinom{n}{2}$,
so we have a decomposition $M=\Rell\allone\oplus Y$ as $\Rell\GL(n,q)$-modules.
Let $a=v_\ell(\qbinom{n-3}{1})$. Then
$v_\ell(\abs{K(\Gamma)})=af$. Using Lemma~\ref{eigenspacebound}
applied to $Y$, we see that $\dim_\Fell(\overline{M_a\cap Y})\geq f$,
so $\dim_\Fell(\overline M_a)\geq f+1$. 
Therefore, by Lemma~\ref{eldiv} with $d=af$, $h=1$, $b_2=1$, $b_1=f+1$ and
$a_1=a$, 
we obtain
\begin{equation}
e_{a}=f,\quad  e_0=g, \quad\text{and $e_i=0$ otherwise.} 
\end{equation}

Case (ii) $\ell\mid s$, $\ell\nmid r$.

a) $\ell\nmid\qbinom{n-1}{1}$.
Let $a=v_\ell(s)$. Then since $\ell\nmid\qbinom{n}{2}$, we have
$v_\ell(\abs{K(\Gamma)})=ag$. Also $M=\Rell\allone\oplus Y$.
Using Lemma~\ref{eigenspacebound} applied to $Y$, we see that 
$\dim_\Fell(\overline{M_a\cap Y})\geq g$, so 
$\dim_\Fell(\overline M_a)\geq g+1$. 
Therefore, by Lemma~\ref{eldiv} with $d=ag$, $h=1$, $b_2=1$, $b_1=g+1$ and
$a_1=a$, 
we obtain
\begin{equation}
e_{a}=g,\quad  e_0=f, \quad\text{and $e_i=0$ otherwise.} 
\end{equation}

b) $\ell\mid\qbinom{n-1}{1}$.
Let $a=v_\ell(\qbinom{n-1}{1})$.  Then $v_\ell(s)=a+b$, where $b=v_\ell(\mu)$.
Then $v_\ell(\abs{K(\Gamma)})=(a+b)g-a$.

If $b=0$, then by Lemma~\ref{eigenspacebound}, $\dim_{\Fell}(\overline M_a)\geq g$.
Therefore, by Lemma~\ref{eldiv} with $d=a(g-1)$, $h=1$, $b_2=1$, $b_1=g$ and 
$a_1=a$, 
we obtain
\begin{equation}
e_{a}=g-1,\quad  e_0=f+1, \quad\text{and $e_i=0$ otherwise.} 
\end{equation}

Suppose $b>0$. Here we have $v_\ell(\abs{K(\Gamma)})=(a+b)g-a$.
By Lemma~\ref{eigenspacebound} we have 
$\dim_{\Fell}\overline M_{a+b}\geq g$,  and we claim that
$\dim_{\Fell}\overline M_{b}\geq g+1$.
By Lemma~\ref{srgeq}, we have $L(L-rI)=s(L-rI)+\mu J$, which shows that
$\Image(L-rI)\subseteq M_b$, hence $\Image(\overline{L-rI})\subseteq\overline M_b$. 
The same equation also shows
that on $\overline M$, $\overline L$ has a quadratic minimal polynomial
with distinct roots $r$ and $0$, so $\overline M=\Ker(\overline{L-rI})\oplus
\Ker(\overline L)$.
The algebraic multiplicity
of $0$ is $g+1$, the sum of the multiplicities of  $s$ and $0$
as eigenvalues of $L$, so $\dim_{\Fell}\Ker(\overline L)=g+1$.
Then, from the decomposition above, we have $\dim_{\Fell}\Image(\overline {L-rI})=g+1$,
and our claim follows.
Therefore, by Lemma~\ref{eldiv} with $d=(a+b)g-a$, $h=2$, $b_3=1$, $b_2=g$, $b_1=g+1$, 
$a_2=a+b$ and $a_1=b$, 
we obtain
\begin{equation}
e_{a+b}=g-1,\quad  e_b=1, \quad  e_0=f, \quad\text{and $e_i=0$ otherwise.} 
\end{equation}

\subsubsection{$\ell\mid q+1$}
As $\ell\mid q+1$, we have for any $m$,  $\ell\mid\qbinom{m}{1}$ if and only if
$m$ is even, so we consider two cases according to the parity of $n$.

Case (i) $n$ even.
\begin{lemma} $\ell\mid r$ iff $\ell\mid\frac{n}{2}$ iff $\ell\mid s$.
\end{lemma}
\begin{proof} As $n$ is even, $\ell$ divides $\qbinom{n}{1}$ and $\qbinom{n-2}{1}$
but divides neither $\qbinom{n-1}{1}$ nor  $\qbinom{n-3}{1}$.
Since $\ell$ divides $q+1$ and  $r=q^2\qbinom{n-3}{1}\qqbinom{n/2}{1}$, it follows that $\ell$ divides $r$ if and only if $\ell$ divides $n/2$.
Since 
$$
s=q\qbinom{n-1}{1}(\qbinom{n-1}{1}-(q+2))/(q+1)=q\qbinom{n-1}{1}(q^3\qqbinom{(n-4)/2}{1}+(q-1)),
$$
and $q\equiv -1\pmod\ell$, we see that $\ell$ divides $s$ if and only if
$\ell$ divides $-(n-4)/2-2=-n/2$.
\end{proof}

Let $a=v_\ell(r)$, $b=v_\ell(s)$. We can assume $a$, $b>0$, or else
$v_\ell(\abs{K(\Gamma)})=0$. Then $v_\ell(\qbinom{n}{2})=a$ and we have
$v_\ell(\abs{K(\Gamma)})=af+bg-a$.

We claim (1) $\dim_{\Fell}\overline M_{a+b}\geq f$ and (2) $\dim_{\Fell}\overline M_b\geq g+1$.
To prove (1), we observe that the equation
$L(L-(r+s)I)=-rsI+\mu J$ implies that $(L-(r+s)I)(Y)\subseteq M_{a+b}$. It is easy
to see that $\overline L(\overline Y)$ is not in the span of $\allone$, so
contains a nontrivial composition factor. Since $[S_2:D_1]\neq 0$ we see
that $\overline L(\overline Y)$ must have a composition factor $D_1$, whose
dimension is $f-1$. Also $\Fell\allone\subseteq \overline M_{a+b}$ so (1) is proved.   
For (2), we note that $S_2$ has composition factors $D_2$ and $D_1$, so
by Lemma~\ref{decompmap}, $\overline{M\cap\tilde S_2}$ also has 
these composition factors, so $\overline M_b$ has at least these composition
factors. In addition we have  $\Fell\allone\subseteq \overline M_b$, so 
$\dim_{\Fell}\overline M_b\geq g+1$.
Therefore, by Lemma~\ref{eldiv} with $d=af+bg-a$, $h=2$, $b_3=1$, $b_2=f$, $b_1=g+1$,
$a_2=a+b$ and $a_1=b$, 
we obtain
\begin{equation}
e_{a+b}=f-1,\quad  e_b=g-f+1,\quad  e_0=f, \quad\text{and $e_i=0$ otherwise.} 
\end{equation}

Case (ii) $n$ odd. 
Let $a=v_\ell(r)$, $b=v_\ell(s)$. 
We can assume $v_\ell(\abs{K(\Gamma)})>0$.
By Lemma~\ref{NumThyL}, exactly one of $a$ and $b$ is nonzero.

a) $\ell\nmid\frac{n-1}{2}$. In this case $\ell\nmid\qbinom{n}{2}$, so
$v_\ell(\abs{K(\Gamma)})=af+bg$ and we have a decomposition
$M=\Rell\allone\oplus Y$.

If $a>0$, $b=0$, then from Lemma~\ref{eigenspacebound}, we have
$\dim_{\Fell}\overline{M_a\cap Y}\geq f$, so $\dim_{\Fell}\overline{M_a}\geq f+1$.
Therefore, by Lemma~\ref{eldiv} with $d=af+bg$, $h=1$, $b_2=1$, $b_1=f+1$ and $a_1=a$, 
we obtain
\begin{equation}
e_{a}=f,\quad  e_0=g, \quad\text{and $e_i=0$ otherwise.} 
\end{equation}

If $a=0$, $b>0$, then from Lemma~\ref{eigenspacebound}, we have
$\dim_{\Fell}\overline{M_b\cap Y}\geq g$,  so $\dim_{\Fell}\overline{M_b}\geq g+1$. 
Therefore, by Lemma~\ref{eldiv} with $d=af+bg$, $h=1$, $b_2=1$, $b_1=g+1$ and $a_1=b$, 
we obtain
\begin{equation}
e_{b}=g,\quad  e_0=f, \quad\text{and $e_i=0$ otherwise.} 
\end{equation}

b) $\ell\mid\frac{n-1}{2}$. 
Here we have $b=v_\ell(s)=v_\ell\left(\frac{\qbinom{n-1}{1}}{q+1}\right)=v_\ell(\qbinom{n}{2})$.
We are assuming that $\ell\mid\frac{n-1}{2}$, so since $\ell\mid q+1$,
we have $\qqbinom{\frac{n-1}{2}}{1}\equiv\frac{n-1}{2}\pmod\ell$, so 
$\ell\mid\qqbinom{\frac{n-1}{2}}{1}=\frac{\qbinom{n-1}{1}}{q+1}$
and we have $b>0$. Then by Lemma~\ref{NumThyL} $a=0$.
Then $v_\ell(\abs{K(\Gamma)})=bg-b$.
By Lemma~\ref{eigenspacebound}, we have $\dim_{\Fell}\overline M_b\geq g$.
Therefore, by Lemma~\ref{eldiv} with $d=bg-b$, $h=1$, $b_2=1$, $b_1=g$ and $a_1=b$, 
we obtain
\begin{equation}
e_{b}=g-1,\quad  e_0=f+1, \quad\text{and $e_i=0$ otherwise.} 
\end{equation}


\section{Skew lines, $A$}
\subsubsection{Parameters} 
\begin{equation}
\begin{aligned}
k&=q^4\qbinom{n-2}{2}, \\
\mu&=q^3\qbinom{n-3}{1}
\left(\frac{q^{n-1}-q^2-q+1}{q^2-1}\right)\\
&=q^3\qbinom{n-3}{1}
\left(\frac{\qbinom{n-1}{1}-(q+2)}{q+1}\right),\\
r_A&=-q^2\qbinom{n-3}{1},\quad s_A=q.
\end{aligned}
\end{equation}

In this section we shall write $r$ for $r_{A}$ and $s$ for $s_{A}$.
The multiplicity of $r$ is  $f=\dim_{\Qell}\tilde S_1=\qbinom{n}{1}-1$ and
that of $s$ is $g=\dim_{\Qell}\tilde S_2=\qbinom{n}{2}-\qbinom{n}{1}$.  

\begin{equation}
\begin{aligned}
\abs{S(\Gamma)}&= q^{4+2f+g}\qbinom{n-2}{2}(\qbinom{n-3}{1})^f\\
&=q^{4+2f+g}\left(\frac{\qbinom{n-2}{1}(\qbinom{n-3}{1})^{f+1}}{q+1}\right).
\end{aligned}
\end{equation}

\subsubsection{$\ell\nmid q+1$}
If  $\ell\nmid\qbinom{n-2}{1}$ and $\ell\nmid\qbinom{n-3}{1}$
the $\ell$-part of $S(\Gamma)$ is trivial, so we consider the
nontrivial cases.

Case (i) $\ell\mid\qbinom{n-2}{1}$, $\ell\nmid\qbinom{n-3}{1}$.
Let $a=v_\ell(\qbinom{n-2}{2})$. Then as we have $\allone\in M_a$,  
we have $\dim_{\Fell}(\overline M_a)\geq 1$ and by
Lemma~\ref{eldiv} with $d=a$, $h=1$, $b_2=0$, $b_1=1$ and $a_1=a$, 
we obtain
\begin{equation}
e_{a}=1,\quad  e_0=f+g, \quad\text{and $e_i=0$ otherwise.} 
\end{equation}

Case (ii) $\ell\nmid\qbinom{n-2}{1}$, $\ell\mid\qbinom{n-3}{1}$. 
Let $a=v_\ell(\qbinom{n-3}{1})$. We claim  $\dim_{\Fell}\overline M_a\geq f+1$.
Suppose first that $\ell\nmid\qbinom{n}{1}$.
Then $\ell\nmid\qbinom{n}{2}$, so we have a decomposition $M=\Rell\allone\oplus Y$.
The $r$-eigenspace of $A$ is contained in $Y_{\Qell}$, so 
$f\leq\dim_{\Fell}\overline{M_a\cap Y}$. Also $\allone\in M_a$, so the claim 
is true in this case. Now assume that $\ell\mid\qbinom{n}{1}$. In this case,
$\Fell^{\LL_1}$ is uniserial with its unique simple submodule
spanned by the all-one vector $\allone_1$ of $\Fell^{\LL_1}$. The map 
$\eta:\Fell^{\LL_1}\to M$ maps $\allone_1$ to $(q+1)\allone\neq 0$,
so it is an injective map. Its $(f+1)$-dimensional image is the span
of the elements $[W_p]$. Direct computation shows that
$A([W_p])= q^2\qbinom{n-3}{1}(\allone-[W_p])$, so $\Image(\eta)\subseteq M_a$.
Therefore, by Lemma~\ref{eldiv} with $d=a(f+1)$, $h=1$, $b_2=0$, $b_1=f+1$ and $a_1=a$, 
we obtain
\begin{equation}
e_{a}=f+1,\quad  e_0=g, \quad\text{and $e_i=0$ otherwise.} 
\end{equation}

\subsubsection{$\ell\mid q+1$}

Case (i) $\ell\mid\qbinom{n-2}{1}$, but $\ell\nmid\qqbinom{\frac{n-2}{2}}{1}$.
In this case $\ell\nmid\qbinom{n-2}{2}$ and $\ell\nmid\qbinom{n-3}{1}$, so
the $\ell$-part of $S(\Gamma)$ is trivial.

Case (ii) $\ell\mid\qbinom{n-2}{1}$, and $\ell\mid\qqbinom{\frac{n-2}{2}}{1}$.
In this case, $\ell\nmid \qbinom{n-3}{1}$. Let $a=v_\ell(\qbinom{n-2}{2})$.
Then $\allone\in M_a$, so by Lemma~\ref{eldiv} with $d=a$, $h=1$, $b_2=0$, $b_1=1$ and 
$a_1=a$, 
we obtain
\begin{equation}
e_{a}=1,\quad  e_0=f+g, \quad\text{and $e_i=0$ otherwise.} 
\end{equation}

Case (iii) $\ell\mid\qbinom{n-3}{1}$, but $\ell\nmid\qqbinom{\frac{n-3}{2}}{1}$.
Let $a=v_\ell(\qbinom{n-3}{1})$. Then  $a=v_\ell(q+1)$ and so
$\ell\nmid\qbinom{n-2}{2}=\frac{\qbinom{n-2}{1}\qbinom{n-3}{1}}{q+1}$. We have $f\leq\dim_{\Fell}\overline M_a$ by Lemma~\ref{eigenspacebound}. 
Therefore, by Lemma~\ref{eldiv} with $d=af$, $h=1$, $b_2=0$, $b_1=f$ and $a_1=a$, 
we obtain
\begin{equation}
e_{a}=f,\quad  e_0=g+1, \quad\text{and $e_i=0$ otherwise.} 
\end{equation}

Case (iv) $\ell\mid\qbinom{n-3}{1}$, and $\ell\mid\qqbinom{\frac{n-3}{2}}{1}$.
Let $a=v_\ell(q+1)$ and $b=v_\ell(\frac{\qbinom{n-3}{1}}{q+1})$.
Then $b=v_\ell(\qbinom{n-2}{2})$, so $\allone\in M_b$, and 
$\Fell\allone\subseteq\overline M_b$. Now $\ell\mid\frac{n-3}{2}$, so $\ell\nmid\frac{n-1}{2}$, and it follows that $\ell\nmid\qbinom{n}{2}$.
Thus, $M=\Rell\allone\oplus Y$, where $Y_\Qell$ contains
the whole $r$-eigenspace in $M_\Qell$.
As $v_\ell(r)=a+b$ it follows from Lemma~\ref{eigenspacebound} that 
$f\leq\dim_{\Fell}\overline{M_{a+b}\cap Y}$. 
As $\overline M=\Fell\allone\oplus\overline Y$, and $\allone\in M_{b}$, we also have
$\dim_{\Fell}\overline M_b\geq 1+f$.
Therefore, by Lemma~\ref{eldiv} with $d=f(a+b)+b$, $h=2$, $b_3=0$,
$b_2=f$, $b_1=f+1$, $a_2=a+b$ and $a_1=b$, 
we obtain
\begin{equation}
e_{a+b}=f,\quad e_{b}=1,\quad  e_0=g, \quad\text{and $e_i=0$ otherwise.} 
\end{equation}

\bibliographystyle{amsplain}
\bibliography{grassmannbib}

\section{Appendix}
The two tables in this appendix show the $\Fell\GL(n,q)$-submodule structure of the
permutation module $\overline M$, according to the relation of $\ell$, $q$ and $n$.

\begin{table}[H]
\begin{tabular}{|c|c|c|c|}
\hline
&&$\ell\nmid\qbinom{n-2}{1}$&$\ell\mid\qbinom{n-2}{1}$\\
\hline
\multirow{2}{*}{$\ell\nmid q+1$}&$\ell\nmid\qbinom{n-1}{1}$ &$\begin{matrix}S_2=D_2\\
{\overline M}=\Fell\oplus D_1\oplus D_2\end{matrix}$ &$\begin{matrix}S_2=D_2+D_1\\
{\overline M}=\Fell\oplus{\begin{matrix}\\D_1\\D_2\\D_1\\\end{matrix}
  }\end{matrix}$ \\
\cline{2-4} 
 &$\ell\mid\qbinom{n-1}{1}$&$\begin{matrix}S_2=D_2+\Fell\\
{\overline M}=D_1\oplus\begin{matrix}\Fell\\D_2\\\Fell\end{matrix}
\end{matrix}$ 
& N/A\\
\hline
\hline
\multirow{2}{*}{$\ell\mid q+1$}&$\ell\nmid\lfloor\frac{n-1}{2}\rfloor$&$\begin{matrix}S_2=D_2\\{\overline M}=\Fell\oplus D_1\oplus D_2\end{matrix}$ & N/A \\
\cline{2-4}
 &$\ell\mid\lfloor\frac{n-1}{2}\rfloor$&$\begin{matrix}S_2=D_2+\Fell\\
{\overline M}=D_1\oplus\begin{matrix}\Fell\\D_2\\\Fell\end{matrix}
\end{matrix}$  & N/A\\
\hline
\hline
\end{tabular}
\caption*{Table A. $\ell\nmid\left[\begin{smallmatrix}n\\1
      \end{smallmatrix}\right]_q$.}
\end{table}

\begin{table}[H]
\begin{tabular}{|c|c|c|c|c|}
\hline
&&$\ell\nmid\qbinom{n-2}{1}$&\multicolumn{2}{|c|}{$\ell\mid\qbinom{n-2}{1}$}\\
\hline
\multirow{2}{*}{$\ell\nmid q+1$}&$\ell\nmid\qbinom{n-1}{1}$ &$\begin{matrix}S_2=D_2\\{\overline M}=\begin{matrix}\Fell\\D_1\\\Fell\end{matrix}\oplus D_2
\end{matrix}
$ 
&\multicolumn{2}{|c|}{N/A} \\
\cline{2-5} 
 &$\ell\mid\qbinom{n-1}{1}$&N/A &\multicolumn{2}{|c|}{N/A}\\
\hline
\hline
\multirow{3}{*}{$\ell\mid q+1$}&\multirow{2}{*}{$\ell\nmid\lfloor\frac{n-1}{2}\rfloor$}&
\multirow{2}{*}{N/A} & \multicolumn{2}{|c|}{$S_2=D_2+D_1$} \\
\cline{4-5}
 && &$\begin{matrix}\ell\nmid\qbinom{n}{2}\\
\xymatrix@C=8pt@R=8pt{
&&&D_1\ar@{-}[dl]\ar@{-}[dr]&\\
{\overline M}=&\Fell\oplus&D_2\ar@{-}[dr]&&\Fell\ar@{-}[dl]\\
&&&D_1&}
 \end{matrix}
 $
&
$\begin{matrix}\ell\mid\qbinom{n}{2}\\
\xymatrix@C=8pt@R=8pt{
&&D_1\ar@{-}[dl]\ar@{-}[d]&\\
{\overline M}=&\Fell&D_2\ar@{-}[d]&\Fell\ar@{-}[dl]\\
&&D_1&}
\end{matrix}
$\\
\cline{2-5}
 &$\ell\mid\lfloor\frac{n-1}{2}\rfloor$&N/A & \multicolumn{2}{|c|}{$\begin{matrix}S_2=D_2+D_1+\Fell\\
{\overline M}=\Fell\oplus\begin{matrix}D_1\\\Fell\\D_2\\\Fell\\D_1\end{matrix}
\end{matrix}
$
}\\
\hline
\hline
\end{tabular}
\caption*{Table B. $\ell\mid\left[\begin{smallmatrix}n\\1
      \end{smallmatrix}\right]_q$.}
\end{table}

\end{document}